\documentclass[a4paper,leqno,11pt]{amsart}

\usepackage[a4paper, margin=1in]{geometry}
\usepackage{epsf,graphicx,epsfig}
\usepackage{amscd}
\usepackage{amsmath,latexsym,amssymb,amsthm}
\usepackage[nospace,noadjust]{cite}
\usepackage{textcomp}
\usepackage{setspace,cite}
\usepackage{lscape,fancyhdr,fancybox}
\usepackage[all,cmtip]{xy}
\usepackage{tikz}
\usepackage{tikz-cd}
\usetikzlibrary{shapes,arrows,decorations.markings}
\usepackage{graphicx}

\usepackage{color}
\usepackage{url}
\usepackage{enumerate}
\usepackage[mathscr]{euscript}

\newtheorem{thm}{Theorem}[section]
\newtheorem{lem}[thm]{Lemma}
\newtheorem{prop}[thm]{Proposition}
\newtheorem{cor}[thm]{Corollary}
\theoremstyle{definition}
\newtheorem{defn}[thm]{Definition}
\newtheorem{exa}[thm]{Example}

\theoremstyle{remark}
\newtheorem{rem}[thm]{Remark}
\numberwithin{equation}{section}

\title{}
\author{}
\date{}
\usepackage{amssymb}

\usepackage{hyperref}
\hypersetup{
	colorlinks,
	citecolor=blue,
	filecolor=black,
	linkcolor=blue,
	urlcolor=black
}

\begin{document}
\title{LIE-YAMAGUTI ALGEBRA BUNDLE}

\author{Saikat Goswami}
\email{saikatgoswami.math@gmail.com}
\address{RKMVERI, Belur, Howrah-711202, India.}
\address{TCG Centres for Research and Education in Science and Technology, Institute for Advancing Intelligence, Sector V, Salt lake, Kolkata 700091, INDIA.}

\author{Goutam Mukherjee}
\email{goutam.mukherjee@tcgcrest.org; gmukherjee.isi@gmail.com}
\address{TCG Centres for Research and Education
in Science and Technology, Institute for Advancing Intelligence, Sector V, Salt lake, Kolkata 700091, INDIA.}


\thispagestyle{empty}

\begin{abstract}
	We introduce the notion of Lie-Yamaguti algebra bundle, and show that such bundles appeared naturally from geometric considerations in the work of M. Kikkawa. This motivates us to introduce this object in the proper mathematical framework.  We define  cohomology groups of such bundles with coefficients in a representation extending the definition of cohomology groups of Lie-Yamaguti algebras.
	
\end{abstract}

\vspace{-2cm}
\maketitle

\vspace{-0.7cm}
\section{Introduction}\label{$1$}
Vector bundles play a crucial role in differential geometry and in various applications of geometry in theoretical physics. In addition, if the fibres of vector bundles have some special type of algebraic structure (generally called as algebra bundles), reflecting some special geometric features of the underlying base manifolds then, the study of such algebra bundles turn out to be very useful in characterizing many relevant properties which arise in geometry and physics. The aim of this article is to introduce and study one such algebra bundle having Lie-Yamaguti algebra structure on its fibres, that generalizes the notion of a Lie algebra bundle. Such bundles appeared implicitly in the work of M. Kikkawa to characterize some local geometric properties to investigate a question  originally posed by K. Nomizu, which motivates us to introduce this object in the proper mathematical framework. 

\medskip
Recall that the notion of Lie algebra bundle was introduced by A. Douady and M. Lazard \cite{DL}.  We refer \cite{CP} for examples of Lie algebra bundles. Various aspects of Lie algebra bundles were extensively studied in \cite{Burdujan, Kiranagi, KMA, KPC}). For some recent research on Lie and other algebra bundles we refer \cite{Kumar21, Kumar22, Kumar23, PRK}.  

\medskip
Lie-Yamaguti algebras are generalizations of Lie triple systems and Lie algebras. Triple systems in algebra may be traced back to the works of P. Jordan, J. v. Neumann and E. Wigner \cite{JNW} in quantum mechanics, and N. Kemmer \cite{NK39, NK43} in particle physics. The notion of  Lie triple system was formally introduced as an algebraic object by N. Jacobson \cite{NJ} in connection with problems which arose from quantum mechanics. 

\medskip
K. Nomizu \cite{KN} proved that affine connections with parallel torsion and curvature are locally equivalent to invariant connections on reductive homogeneous spaces, and that each such space has a canonical connection for which parallel translation along geodesics agrees with the natural action of the group. 

\medskip
Let $M$ be a smooth manifold equipped with a linear connection $\nabla.$ Let $e\in M$ be a given fixed point. Then there is a local multiplication $\mu$ at $e$ compatible with $\nabla$, which is given by
$$\mu (x, y) = exp_x \circ \tau_{e,x} \circ exp_e^{-1}(y),$$ where $exp_x$ denotes the exponential mapping at $x$ and $\tau_{e,x}$ denotes the parallel displacement of tangent vectors along the geodesic joining $e$ to $x$ in a normal
neighbourhood of $e$  \cite{MK75}. 

\medskip
If $M = A/K$ is a reductive homogeneous space with the canonical connection, due to K. Nomizu, then the local multiplication $\mu$ given above satisfies some special property (cf. \cite{KN}). In particular, if M is a Lie group $A$
itself, then the canonical connection is reduced to the connection of \cite{CS} and the local multiplication $\mu$ coincides with the multiplication of $A$ in local. 

\medskip
Motivated by this fact, M. Kikkawa \cite{MK75} investigated the problem of the existence of a global differentiable binary system on a reductive homogeneous space $A/K,$ which coincides locally with the above geodesic local multiplication $\mu$ and observed that the problem is related to the canonical connection and to the general Lie triple system defined on the tangent space $T_eM.$ In his paper, Kikkawa renamed the notion of general Lie triple system as {\it Lie triple algebra}. Kinyon and Weinstein \cite{KW} observed that Lie triple algebras, which they called Lie-Yamaguti algebras in their paper, can be constructed from Leibniz algebras. Leibniz algebras are non anti-symmetric analogue of Lie algebras introduced by J. L. Loday \cite{Loday93}.

\medskip
In the present article, we introduce the notion of a Lie-Yamaguti algebra bundle, and discuss examples of such bundles. We show that such bundles appeared implicitly in the work of M. Kikkawa \cite{MK99}. We define cohomology groups of a Lie-Yamaguti algebra bundle with coefficients in a given representation. Our theory extends the cohomology theory of  Lie-Yamaguti algebras introduced by K. Yamaguti in \cite{KY-cohomology}.  


\medskip\noindent
{\bf Organization of the paper}: In \S \ref{$2$}, we set up notations, recall some known definitions and results. In  \S \ref{$3$}, we introduce the main object of study of the present paper, namely, the notion of a {\it Lie-Yamaguti algebra bundle}, illustrate examples of such bundles and describe a general method of constructing such bundles. In  \S \ref{$4$}, we introduce the concept of representation of Lie-Yamaguti algebra bundles which is required to introduce cohomology of Lie-Yamaguti algebra bundles. Finally in  \S \ref{$5$}, we define a cochain complex which defines cohomology of a Lie-Yamaguti algebra bundle with coefficients in a given representation. 

\section{Preliminaries}\label{$2$}

The aim of this section is to recall some basic definitions and set up notations to be followed throughout the paper. Let $\mathbb K$ be a given field.
\begin{defn}
	A Lie algebra is a vector space $\mathfrak g$ over $\mathbb K$ equipped with a $\mathbb K$-bilinear operation $[~,~] : \mathfrak g\times \mathfrak g \rightarrow \mathfrak g$ satisfying
	\begin{enumerate}
		\item (Anti-symmetry): $[x, y] = -[y, x]$ for all $x, y \in \mathfrak g$;
		\item (Jacobi identity): $[[x,y],z] + [[y,z],x] + [[z,x],y] =0$  for all $x, y, z \in \mathfrak g.$ 
	\end{enumerate}
\end{defn}

\begin{defn}
	A Leibniz algebra is a vector space $\mathfrak g$ over $\mathbb K$ equipped with a $\mathbb K$-bilinear operation $\cdot : \mathfrak g \times \mathfrak g \rightarrow \mathfrak g$ satisfying the Leibniz identity
	$$ x\cdot (y\cdot z) = (x\cdot y)\cdot z + y\cdot (x\cdot z)$$
	for all $x, y, z \in \mathfrak g.$
\end{defn}
It is easy to see that in presence of the anti-symmetric condition the Leibniz identity reduces to Jacobi identity. Thus, Lie algebras are examples of Leibniz algebras. See \cite{Loday93} for many other non-trivial examples of Leibniz algebras.

\begin{defn}
	A Lie triple system is a vector space $\mathfrak g$ over $\mathbb K$ equipped with a $\mathbb K$-trilinear operation 
	$$\{~,~,~\} : \mathfrak g\times \mathfrak g\times \mathfrak g \rightarrow \mathfrak g$$ satisfying
	\begin{equation*}
		\{ x, y, z\} = -\{ y, x, z\}
	\end{equation*}
	\begin{equation*}
		\{ x,y,z\} + \{ y,z,x\} + \{ z,x,y\} = 0
	\end{equation*}
	\begin{equation*}
		\{ x, y, \{ u, v, w\} \} 
		= \{ \{ x, y, u \} , v, w\} 
		+ \{ u, \{ x, y, v\} , w\} + \{ u, v, \{ x, y, w\} \} 
	\end{equation*}
	for all $x, y, u, v, w \in \mathfrak g.$
\end{defn}

The following is an interesting example of a Lie triple system which arose from Physics \cite{NJ}.

\begin{exa}\label{Meson}
	We denote by $M_{n+1}(\mathbb R)$, the set of all $(n+1) \times (n+1)$ matrices over the field $\mathbb R$, which is an associative algebra with respect to matrix multiplication.  Let $\delta_{ij}$ denote the Kronecker delta symbol
	\[ \delta_{ij}= \left\{\begin{array}{ll}
		0 & i \neq j \\
		1 & i=j 
	\end{array}
	\right.\]
	and $e_{i,j}$ denote the elementary matrix which has $1$ in the $(i,j)^{\mbox{th}}$-entry as its only non-zero entry.
	Let $\mathfrak m$ be the subspace of $M_{n+1}(\mathbb R)$ spanned by the matrices $G_i$ for $i=1,2,\cdots,n$, where $G_i = e_{i,n+1}-e_{n+1, i}.$ 
	\vspace{0.1in}
	As an example, for $n=3,$ the matrix $G_2 \in M_4(\mathbb R)$ is given by
	\[
	G_2 =
	\begin{pmatrix}
		0 & 0 & 0 & 0 \\
		0 & 0 & 0 & 1 \\
		0 & 0 & 0 & 0 \\
		0 & -1 & 0 & 0 \\
	\end{pmatrix}.
	\]
	\noindent
	Then, the subspace $\mathfrak m$ is closed under the ternary product 
	$$\{A, B, C\}:=[[A,B],C], ~~A, B, C \in \mathfrak g$$ where $[A,B] := AB - BA$ is the commutator bracket. Explicitly, the trilinear product of the basis elements are given by
	\[
	[[G_i,G_j],G_k] = \delta_{ki} G_j - \delta_{kj} G_i.
	\] 
	It turns out that $(\mathfrak m, \{~, ~,~\})$ is a Lie triple system, first used in \cite{Duffin} to provide a significant and elegant algebraic formalism of Meson equations and hence was known as Meson field. Later, it was introduced formally as a Lie triple system by N. Jacobson in \cite{NJ}.
\end{exa}
\begin{rem}
	Note that any Lie algebra $(\mathfrak g, [~,~])$ can be viewed as a Lie triple system with the trilinear operation
	$$\{ x, y, z\} := [[x, y], z]$$  for all $x, y, z \in \mathfrak g.$
\end{rem}

\begin{defn}
	A Lie-Yamaguti Algebra $(\mathfrak g, [~,~], \{~, ~, ~\})$ is a vector space $\mathfrak g$ equipped with a $\mathbb K$-bilinear and a trilinear operation 
	$$[~, ~]: \mathfrak g \times \mathfrak g \rightarrow \mathfrak g \quad \text{and} \quad 
	\{~, ~, ~\} : \mathfrak g\times \mathfrak g \times \mathfrak g \rightarrow \mathfrak g$$ 
	such that for all $x, y, z, u, v, w \in \mathfrak g$ the following relations hold:
	\begin{equation}
		[x, y] = -[y,x]; \tag{LY1} \label{LY1} 
	\end{equation}
	\begin{equation}
		\{x, y, z\} = -\{y,x, z\}; \tag{LY2} \label{LY2}
	\end{equation}
	\begin{equation}
		\Sigma_{\circlearrowleft (x, y, z)}([[x,y], z] + \{x, y, z\}) =0;  \tag{LY3} \label{LY3}
	\end{equation}
	\begin{equation}
		\Sigma_{\circlearrowleft (x, y, z)}\{[x, y], z, u\} = 0;  \tag{LY4} \label{LY4}
	\end{equation}
	\begin{equation}
		\{x, y, [u, v]\} = [\{x, y, u\}, v] + [u, \{x, y, v\}];  \tag{LY5} \label{LY5}
	\end{equation}
	\begin{equation}
		\{x, y, \{u, v, w\}\}
		= \{\{x, y, u\}, v, w\} + \{u, \{x, y, v\}, w\} + \{u, v, \{x, y, w\}\}.  \tag{LY6} \label{LY6} 
	\end{equation}
	Here, $\Sigma_{\circlearrowleft (x,y,z)}$ denotes the sum over cyclic permutations of x, y, and z.
\end{defn}

\begin{rem}
	Notice that if the trilinear product in a Lie-Yamaguti algebra is trivial, that is, if $\{~,~,~\} = 0,$ then (LY$2$), (LY$4$), (LY$5$), and (LY$6$) are trivial, and (LY$1$) and
	(LY$3$) define a Lie algebra structure on $\mathfrak g$. On the other hand, if the binary product is trivial, that is, $[~,~] =0,$ then (LY$1$), (LY$4$), and (LY$5$) are trivial, and (LY$2$), (LY$3$), together with (LY$6$) define a Lie triple system on $\mathfrak g.$
\end{rem}

Here are some well-known examples. 

\begin{exa}\label{Lie to LYA}
	Let $(\mathfrak g, [~, ~])$ be a Lie algebra over $\mathbb K$. Then, $\mathfrak g$ has a Lie-Yamaguti algebra structure induced by the given Lie bracket, the trilinear operation being:
	$$\{a, b, c\} = [[a, b], c]$$ for all $a, b, c \in \mathfrak g$.
\end{exa}

\begin{exa}
	Let $(\mathfrak g, \cdot)$ be a Leibniz algebra. Define a bilinear operation and a trilinear operation as follows:
	$$[~, ~]: \mathfrak g \times \mathfrak g \to \mathfrak g, ~~ [a, b] := a\cdot b-b\cdot a,~~ a, b \in \mathfrak g;$$
	$$\{~, ~, ~\}:  \mathfrak g \times \mathfrak g \times \mathfrak g \to \mathfrak g, ~~\{a, b, c\} := -(a\cdot b)\cdot c,~~a, b, c \in \mathfrak g.$$ Then, $(\mathfrak g, [~,~],\{~,~,~\})$ is a Lie-Yamaguti algebra. 
\end{exa}
Let $(\mathfrak g, \langle~,~\rangle)$ be a Lie algebra. Recall that a reductive decomposition of
$\mathfrak g$ is a vector space direct sum $\mathfrak g = \mathfrak h \oplus\mathfrak m$ satisfying $\langle\mathfrak h, \mathfrak h\rangle \subseteq\mathfrak h$ and $\langle\mathfrak h, \mathfrak m\rangle \subseteq \mathfrak m.$ In this case, we call $(\mathfrak h, \mathfrak m)$ a {\it reductive pair}.

\begin{exa}\label{reductive}
	Let $(\mathfrak g, \langle~, ~\rangle)$ be a Lie algebra with a reductive decomposition $\mathfrak g = \mathfrak h \oplus\mathfrak  m.$ Then, there exist a natural binary and a ternary product on $\mathfrak m$ defined by
	$$ [a, b] := \pi_{\mathfrak m}(\langle a, b\rangle),~~\{a, b, c\} := \langle \pi_{\mathfrak h}(\langle a, b\rangle), c \rangle,$$ where $\pi_{\mathfrak m}$ and $\pi_{\mathfrak h}$ are the projections on $\mathfrak m$ and $\mathfrak h$, respectively. These products endow $\mathfrak m$ with the structure of a Lie-Yamaguti algebra \cite{BBM}.
\end{exa}

\begin{exa}
	Consider the vector space $\mathfrak g$ over $\mathbb K$ generated by $\{e_1, e_2, e_3\}.$ Define a bilinear operation $[~,~]$ and a trilinear operation $\{~, ~, ~\}$ on $\mathfrak g$ as follows.
	$$ [e_1, e_2] = e_3;~~\{e_1, e_2, e_1\} = e_3.$$
	All other brackets of the basis elements are either determined by the definition of Lie-Yamaguti algebra or else are zero. Then, $\mathfrak g$ with the above operations is a Lie-Yamaguti algebra.
\end{exa}

See \cite{ABCO} for classification of some low dimensional Lie-Yamaguti algebras.

\begin{defn}
	Let $(\mathfrak g, [~,~], \{~,~,~\})$,  $(\mathfrak g^\prime, [~,~]^\prime, \{~,~,~\}^\prime)$ be two Lie-Yamaguti algebras. A homomorphism
	$$ \phi :  (\mathfrak g, [~,~], \{~,~,~\}) \rightarrow (\mathfrak g^\prime, [~,~]^\prime, \{~,~,~\}^\prime)$$ of Lie-Yamaguti algebras is a $\mathbb K$-linear map $\phi :  \mathfrak g \rightarrow \mathfrak g^\prime$ satisfying
	$$\phi ([x, y]) = [\phi (x), \phi (y)]^\prime,~~~~\phi (\{x, y, z\}) = \{\phi (x), \phi (y), \phi (z)\}^\prime$$ for all $x, y, z \in \mathfrak g.$
	
	A homomorphism
	$$ \phi :  (\mathfrak g, [~,~], \{~,~,~\}) \rightarrow (\mathfrak g^\prime, [~,~]^\prime, \{~,~,~\}^\prime)$$ of Lie-Yamaguti algebras is an isomorphism if there exists a homomorphism
	$$\phi^\prime:  (\mathfrak g^\prime, [~,~]^\prime, \{~,~,~\}^\prime) \rightarrow (\mathfrak g, [~,~], \{~,~,~\})$$ such that $\phi^\prime \circ \phi = id_{\mathfrak g}$ and $\phi \circ \phi^\prime = id_{\mathfrak g^\prime}.$  The set of all self-isomorphisms of a Lie-Yamaguti algebra  $(\mathfrak g, [~,~], \{~,~,~\})$ is obviously a group under composition of maps and is denoted by ${\mbox{Aut}}_{\mbox{LYA}}(\mathfrak g).$
\end{defn}

The notion of Lie algebra bundle was introduced in \cite{DL}. For smooth Lie algebra bundle we refer \cite{mackenzie}. Other notions of algebra bundles are available in the literature and appeared in various context. 

\medskip
Let $M$ be a smooth manifold (Hausdorff and second countable, hence, paracompact). Let $ C^\infty(M)$ be the algebra of smooth functions on $M$. Let $TM$ be the tangent bundle of $M$. Recall that a vector field on $M$ is a smooth section of the tangent bundle $TM.$ Let us denote the space of  vector fields on $M$ by $\chi (M).$ It is well-known that $\chi (M)$ is a  $C^\infty (M)$-module. Moreover, $\chi (M)$ is a Lie algebra with the commutator bracket: 
$$[\alpha, \beta] := \alpha \beta -\beta \alpha$$ for $\alpha, \beta \in \chi (M).$ Here, for $\alpha, \beta \in \chi (M)$ and $p \in M,$ the action of $ \alpha \beta (p)$ on a smooth function $f \in C^\infty (M)$ is given by 
$$ \alpha \beta (p)(f) = \alpha_p(\beta f),$$ where $\beta f \in C^\infty(M)$ is given by $\beta f (m) = \beta_m (f), ~~m \in M.$

\medskip
For a (smooth) vector bundle $p : L \rightarrow M,$ often denoted by $\xi = (L, p, M),$ we denote the space of smooth sections of $L$ by $\Gamma L.$  It is well-known that $\Gamma L$ is a  $C^\infty (M)$-module. For any $m\in M,$ we denote the fibre of the vector bundle $\xi$ over $m$ by $L_m$ or sometimes by $\xi_m.$

\medskip
Henceforth, we will work in the smooth category and with $\mathbb K = \mathbb R.$
\begin{defn}
	Let $(L, p, M)$ be a vector bundle and let $[~, ~]$ be a smooth section of the bundle $\mbox{Alt}^2(L)$ such that for each $m \in M,$
	$$[~,~]_m : L_m \times L_m   \rightarrow L_m $$ is a Lie algebra bracket on $L_m.$  We call such a section a field of Lie algebra brackets in $L.$ 
\end{defn}

\begin{defn}
	A Lie algebra bundle (cf. \cite{mackenzie}) is a vector bundle $(L, p, M)$ together with a field of Lie algebra brackets
	$$m \mapsto [~, ~]_m,~m \in M.$$ 
\end{defn}
Thus, for a Lie algebra bundle $(L, p, M),$ each fibre $L_m$ is a Lie algebra which varies smoothly as $m\in M$ varies over $M.$ In other words, the assignment $m\mapsto [~,~]_m,~~m \in M$ is smooth.
\begin{defn}
	Let $\mathfrak g$ be a given Lie algebra. A locally trivial Lie algebra bundle with fibre $\mathfrak  g$ is a vector bundle $(L, p, M)$ together with a field of Lie algebra brackets
	$$m\mapsto [~, ~]_m,~~m \in M$$ 
	such that $M$ admits an open covering $\{U_i\}$ equipped with local trivializations $\{\psi_i: U_i\times \mathfrak g \rightarrow p^{(-1)}(U_i)\}$ for which each $\psi_{i,m},~~m \in M$ ($\psi_i$ restricted to each fibre $L_m$) is a Lie algebra isomorphism.
	
	\medskip
	A homomorphism $\phi: (L, p, M) \rightarrow (L^\prime, p^\prime, M^\prime)$ of Lie algebra bundles is a vector bundle morphism $(\phi, \phi_0),$ where $\phi_0: M\rightarrow M^\prime$ such that $ \phi|_{L_m}: L_m  \rightarrow L^\prime_{\phi_0(m)},~m \in M$ is a Lie algebra homomorphism. 
\end{defn}

\section{Lie-Yamaguti Algebra Bundle}\label{$3$}

In this section, we introduce the notion of Lie-Yamuguti algebra bundle and related results. All vector bundles and vector bundle maps are assumed to be smooth and $\mathbb K = \mathbb R.$

\begin{defn}
	Let $\xi = (L, p, M)$ be a (real) vector bundle. Let $\mbox{Hom}(\xi^{\otimes k}, \xi)$ be the real vector space of vector bundle maps from $\xi^{\otimes k}$ to the vector bundle $\xi,$ $k \geq 1.$ Observe that $\mbox{Hom}(\xi^{\otimes k}, \xi)$ is a vector bundle over $M.$  Let $\langle~, \cdots , ~\rangle$ be a section of the bundle $\mbox{Hom}(\xi^{\otimes k}, \xi).$  We call such a section a $k$-field of ($\mathbb K$-multinear) brackets in $\xi.$ Thus, a $k$-field of brackets in $\xi$ is a smooth assignment 
	$$m\mapsto (\langle~,\cdots, ~\rangle_m : \xi_m \times \cdots \times \xi_m  \rightarrow \xi_m) $$ of multilinear operation on $\xi_m$, $m \in M.$  
\end{defn}

\begin{defn}
	A Lie-Yamaguti algebra bundle is a vector bundle $\xi=(L, p, M)$ together with a $2$-field and a $3$-field of brackets
	$$m\mapsto [~, ~]_m 
	\quad \text{and} \quad 
	m \mapsto \{~,~, ~\}_m,~~m \in M$$
	which make each fibre $\xi_m,$ $m\in M$ a Lie-Yamaguti algebra. 
\end{defn}

\begin{defn}
	Let  $(\mathfrak g, [~,~]_{\mathfrak g}, \{~, ~, ~\}_{\mathfrak g})$ be a given Lie-Yamaguti algebra. A locally trivial Lie-Yamaguti algebra bundle is a vector bundle $\xi=(L, p, M)$ together with a $2$-field and a $3$-field of brackets
	$$m\mapsto [~, ~]_m 
	\quad \text{and} \quad 
	m \mapsto \{~,~, ~\}_m,~~m \in M$$
	such that $M$ admits an open covering $\{U_i\}$ equipped with local trivializations $\{\psi_i: U_i\times \mathfrak g \rightarrow p^{(-1)}(U_i)\}$ for which each $\psi_{i,m},~~m \in M$ ($\psi_i$ restricted to each fibre $\xi_m$) is a Lie-Yamaguti  algebra isomorphism.
\end{defn}

\begin{rem}
	Thus, for a Lie-Yamaguti algebra bundle as defined above each fibre $\xi_m= p^{-1}(m),~~m\in M,$ together with the binary operation $[~,~]_m$ and the ternary operation $\{~, ~,~\}_m$ is a Lie-Yamaguti algebra isomorphic to $\mathfrak g,$ and the assignments
	$$m \mapsto [~,~]_m,~~m\mapsto \{~, ~, ~\}_m$$ varies smoothly over $M.$ 
\end{rem}

In other words, a locally trivial Lie-Yamaguti algebra bundle over $M$ is a vector bundle over $M$ such that each fibre of the bundle has a Lie-Yamaguti algebra structure isomorphic to $\mathfrak g.$

\medskip
An obvious example of a Lie-Yamaguti algebra bundle is the trivial bundle over a smooth manifold $M$ with fibres a Lie-Yamaguti algebra.

\begin{exa}
	Let  $(\mathfrak g, [~,~], \{~, ~, ~\})$ be a given Lie-Yamaguti algebra and $M$ be any smooth manifold. Then the trivial vector bundle $\xi = M\times \mathfrak g$ with the projection onto the first factor $ \pi_1: M\times \mathfrak g \to M$ is a Lie-Yamaguti algebra bundle, called the product Lie-Yamaguti algebra bundle.
\end{exa}

We have the following  example from the Example \ref{Lie to LYA}.

\begin{exa}\label{Dorfman}
	Any Lie algebra bundle $(L, p, M, [~,~])$ is a Lie-Yamaguti algebra bundle, where the $3$-field of brackets on $M$ induced by the $2$-field of Lie brackets
	$m \mapsto [~,~]_m, ~~m\in M$ is defined by 
	$$\{a, b, c\}_m := [[a, b]_m, c]_m,~~m\in M,$$ for $a, b, c \in L_m, ~m\in M.$
\end{exa}

\begin{defn}
	Let $\xi = (L, p, M)$ be a Lie algebra bundle with the field of Lie algebra bracket $m \mapsto [~,~]_m,~~m\in M.$ A reductive decomposition of $\xi$ is a pair $(L^1, L^2)$ of sub-bundles of $L$ such that $L$ is a Whitney sum $L = L^1\oplus L^2$ satisfying $[ L^1_m, L^1_m]_m \subseteq L^1_m$ and $[L^1_m, L^2_m]_m \subseteq L^2_m.$ In this case, we call $(L^1, L^2)$ a reductive pair. 
\end{defn}

For a reductive pair as above, let $\pi^i: L\to L^i,~~i= 1, 2$ denote the vector bundle projection maps.

\begin{exa}
	Let $(L^1, L^2)$ be a reductive decomposition of a Lie algebra bundle $\xi = (L, p, M)$ as described in the above definition.  Then, define a $2$-field of brackets and a $3$-field of brackets
	$$m\mapsto \langle~,~\rangle_m,~~ m\mapsto \{~,~,~\}_m,~m \in M$$ on the vector bundle $(L^2, p|_{L^2}, M)$ as follows. Let $a, b, c \in L^2_m,~~m \in M.$ 
	$$ [a, b]_m := \pi^1(\langle a, b \rangle_m),~~\{a, b, c\} := \langle \pi^2(\langle a, b\rangle_m), c \rangle_m.$$ Then, as in the case of Example \ref{reductive}, the vector bundle $(L^2, p|_{L^2}, M)$ is a Lie-Yamaguti algebra bundle equipped with the $2$-field of brackets and the $3$-field of brackets as defined above.
\end{exa}

\begin{exa}\label{derivation-bundle}
	Let $\xi=(E, p, M)$ be a vector bundle with fibre $V.$ Consider the vector bundle $\mbox{End}(\xi) : = \mbox{Hom}(\xi,\xi)$ with fibres 
	$$\mbox{End}(\xi)_m = \mbox{End}(E_m)\cong \mbox{End}(V)= \mbox{Hom}(V, V),~~ m \in M.$$ Note that from the Example \ref{Lie to LYA}, $\mbox{End}(V)$ is a Lie-Yamaguti algebra as it is a Lie algebra with respect to the commutator bracket. The local triviality for $\mbox{End}(\xi)$ are induced from the local triviality of $\xi$ in the following way: 
	
	\medskip
	\noindent
	Any chart $\psi: U \times V \to E_U$ for $E$ induces a chart $\overline{\psi}: U \times \mbox{End}(V) \to \mbox{End}(E)_U$ where for any $m \in M$, $\overline{\psi}: \mbox{End}(V) \to \mbox{End}(E_m) = \mbox{End}(V)$ is defined as $T \mapsto \psi_m \circ T \circ \psi_m^{-1}$. It follows that $\mbox{End}(\xi)$ is a locally trivial Lie-Yamaguti algebra bundle with respect to this charts with fibres isomorphic to $\mbox{End}(V).$ Observe that the $2$-field of brackets are given by $m \mapsto [~, ~]_m, ~m \in M,$  $[~, ~]_m$ being the usual commutator Lie bracket of $\mbox{End}(E_m)$  and the $3$-field of brackets are $m \mapsto \{~, ~, ~\}_m, ~m \in M,$ where 
	$$\{~, ~, ~\}_m := [[~, ~]_m,~]_m.$$
\end{exa}

Given a locally trivial Lie-Yamaguti algebra bundle with fibres isomorphic to a given Lie-Yamaguti algebra, one can obtain many other examples of Lie-Yamaguti algebra bundles out of it. The precise statement is given by the following proposition which is straightforward to prove.

\begin{prop}\label{characteristic_sub bundle}\label{characteristic_subbundle}
	Let $\xi = (L,p,M,[~,~],\{~,~,~\})$ be a locally trivial Lie-Yamaguti algebra bundle with fibers isomorphic to a given Lie-Yamguti algebra $\mathfrak{g}$. Let $\mathfrak{h}$ be a subalgebra of $\mathfrak{g}$ such that $\varphi(\mathfrak{h}) = \mathfrak{h}$ for all $\varphi \in \operatorname{Aut}_{\operatorname{LYA}}(\mathfrak{g})$. Then there is a well-defined locally trivial Lie-Yamaguti algebra sub bundle $\eta$ of $\xi$ with total space  $K \subset L$ such that any Lie-Yamaguti algebra bundle chart $\psi: U \times \mathfrak{g} \to L_U$ of $\xi$ restricts to a Lie-Yamaguti algebra bundle chart $U \times \mathfrak{h} \to K_U$ of $\eta$. 
\end{prop}

As an application of the above result we obtain the following example. 

\medskip
Recall that for a given Lie-Yamaguti algebra $(\mathfrak g, [~,~], \{~, ~, ~ \})$ a linear map $D: \mathfrak g \to \mathfrak g$ is called a derivation of $\mathfrak g$ if for all $x, y, z \in \mathfrak g$ 
$$D([x,y]) = [D(x),y] + [x,D(y)],$$  
$$D\{x,y,z\} = \{D(x),y,z\} + \{x,D(y),z\} + \{x,y,D(z)\}.$$
For any $x,y \in \mathfrak g$ the map $D(x,y): \mathfrak g \to \mathfrak g$ defined by $D(x,y)(z) := \{x,y,z\}$ is a derivation and is called an {\bf inner derivation}. We denote the space of derivations of $\mathfrak g$ by $\operatorname{Der}_{\operatorname{LYA}}(\mathfrak{g}).$ 
\begin{exa}\label{derivation-cohomology}
	Let $(\mathfrak g, [~,~], \{~, ~, ~ \})$ be any Lie-Yamaguti algebra and $\xi = (L, p, M,[~,~],\{~,~,~\})$ be a locally trivial Lie-Yamaguti algebra bundle with fibres isomorphic to the Lie-Yamguti algebra $\mathfrak{g}$. Observe that $\operatorname{Der}_{\operatorname{LYA}}(\mathfrak{g}) \subseteq \operatorname{End}_{\operatorname{LYA}}(\mathfrak{g})$ is invariant under automorphisms of the form $\operatorname{End}_{\operatorname{LYA}}(\mathfrak{g}) \to \operatorname{End}_{\operatorname{LYA}}(\mathfrak{g})$ defined as $T \mapsto s \circ T \circ s^{-1}$ for any Lie-Yamaguti algebra automorphism  $s: \mathfrak{g} \to \mathfrak{g}.$ To see this, we need to show that if $T \in \operatorname{Der}_{\operatorname{LYA}}(\mathfrak{g})$ then $s \circ T \circ s^{-1}$ is also a derivation, that is, we need to show $s \circ T \circ s^{-1}[x,y] = [s \circ T \circ s^{-1} x, y] + [x, s \circ T \circ s^{-1}y]$ for $x, y \in \mathfrak g.$ This is true because for $x, y \in \mathfrak g$
	\begin{eqnarray*}
		sTs^{-1}[x,y] 
		&=& sT [s^{-1}x,s^{-1}y] = s[Ts^{-1}x, s^{-1}y] + s[s^{-1}x,Ts^{-1}y] \\
		&=& [sTs^{-1}x,y] + [x,sTs^{-1}y]. 
	\end{eqnarray*}    
	Similarly, we have $sTs^{-1} \{x,y,z\} = \{sTs^{-1}x,y,z\} + \{x,sTs^{-1}y,z\} + \{x,y,sTs^{-1}z\}$ for $x, y, z \in \mathfrak g.$ Hence, $\operatorname{Der}_{\operatorname{LYA}} (\mathfrak{g})$ is invariant under automorphism of the forms stated above. Applying Proposition \ref{characteristic_subbundle}, we obtain a locally trivial Lie-Yamaguti algebra sub bundle $\mbox{Der}(\xi)$ with fibres isomorphic to $\operatorname{Der}_{\operatorname{LYA}}(\mathfrak{g})$ of the Lie-Yamaguti algebra bundle $\mbox{End}(\xi).$ We call it the Lie-Yamaguti algebra bundle of derivations of $\xi.$  
\end{exa}

Next, we discuss an interesting example of a Lie-Yamaguti algebra bundle that arose from the work of M. Kikkawa \cite{MK64, MK75, MK99} to characterize some local geometric properties. We recall some definitions which are necessary to describe our next example.

\medskip
Recall that a linear connection on a smooth manifold $M$ is an $\mathbb R$-bilinear map 
$$\nabla: \chi (M) \times \chi (M) \to \chi (M)$$ 
written $\nabla_XY$ for $\nabla(X,Y)$, satisfying two properties stated below: 
For all $X,Y \in \chi (M)$ 
\begin{itemize}
	\item $\nabla_XY$ is a $C^{\infty}(M)$-linear in $X$.
	\item (Leibniz rule) $\nabla_XY$ satisfies the Leibniz rule in $Y$: For all $f \in C^{\infty} (M)$,
	$$
	\nabla_X(fY) = (Xf)Y+f(\nabla_XY).
	$$  
\end{itemize}

\noindent
Now, let $M$ be a smooth manifold along with linear connection $\nabla$. 
Recall that 
\begin{itemize}
	\item a torsion tensor of the connection $\nabla$ is a 
	$C^{\infty}(M)$-bilinear map  
	$$S: \chi(M) \times \chi(M) \to \chi(M)$$ defined by 
	$$
	S(X,Y) := \nabla_XY - \nabla_YX - [X,Y],~~~X, Y \in \chi (M),
	$$  
	where $[X,Y]$ is the Lie bracket of $\chi (M)$ and
	\item a  curvature tensor of the connection $\nabla$ is a $C^{\infty}(M)$-trilinear map 
	$$R:\chi (M) \times \chi (M) \times \chi (M) \to \chi (M)$$
	defined by 
	$$
	R(X,Y)Z := \nabla_X \nabla_Y Z - \nabla_Y \nabla_X Z - \nabla_{[X,Y]} Z,~~~X, Y, Z \in  \chi(M).$$
\end{itemize}
Recall the following definitions \cite{MK75}.
\begin{defn}
	Let $M$ be a smooth manifold with a connection $\nabla$. Let $S$ and $R$ denote the torsion and curvature tensors of $\nabla,$ respectively. Then, $(M,\nabla)$ is said to be a locally reductive space if $~\nabla S = 0 ~~\&~~ \nabla R =0;$ that is,
	\begin{itemize}
		\item for all $X,Y,Z \in \chi (M)$;  $\nabla_X S(Y,Z) = 0;$
		\item for all $X, U, V, W \in \chi (M)$; $\nabla_X R(U,V)W = 0.$
	\end{itemize}
\end{defn}

\begin{defn}
	Let $G$ be a connected Lie group and $H$ be a closed subgroup of $G.$ Then the homogeneous space $M = G/H$ is said to be reductive if and only if $G$ acts effectively on $M$ and the Lie algebra $\mathfrak g$ of $G$ admits a direct sum decomposition as 
	$$\mathfrak g = \mathfrak m \oplus \mathfrak h,$$ where $\mathfrak h$ is the Lie algebra of $H$ and $\mathfrak m$ is a subspace of $\mathfrak g.$ 
\end{defn}

Next, we recall the notion of homogeneous Lie loops.

\begin{defn} 
	Let $G=(G,\mu)$ be a binary system with the binary operation 
	$$\mu:G \times G \to G$$
	$G$ is a loop if there is a (two-sided) identity $e \in G$, $xe=ex=x ~~ (x \in G)$, and the left and right translations of $G$ by any element $x \in G$, denoted by 
	$$
	L_x,R_x:G \to G;~L_x(y) = xy,~ R_x(y) = yx ~~~~ (y \in G),
	$$
	are permutations of $G$.
\end{defn}

\begin{defn}  
	A loop $G$ is said to have the left inverse property, if for any $x \in G$ there exists an element $x^{-1} \in G$ such that 
	$$
	x^{-1}(xy) = y ~~~~ (y \in G)
	$$
\end{defn}

\begin{defn}
	Let $L_0(G)$ be the group generated by all left inner mappings, i.e., 
	$$
	L_{x,y} = L_{xy}^{-1} \circ L_x \circ L_y ~~~~ (x,y \in G)
	$$
	A loop $G$ is called a left A-loop, if the left inner mapping group $L_0(G)$ is a subgroup of the automorphism group $AUT(G)$ of $G$.
\end{defn}

\begin{defn}
	A Loop $(G,\mu)$ is said to be a homogeneous loop, if it is a left A-loop with the left inverse property. 
\end{defn}

\begin{defn} 
	A homogeneous Lie loop $G$ is a homogeneous loop, and is also a smooth manifold such that the loop multiplication $\mu: G \times G \to G$ is smooth.
\end{defn}

Here are some examples of locally reductive spaces. 

\noindent
\begin{itemize}
	\item Let $G$ be a connected homogeneous Lie loop equipped with the canonical connection.
	
	\item Define $K(G):=$ the closure of $L_0(G)$ in the smooth automorphism group $Aut(G)$ of $G$, and consider the semi-direct product $A(G)=G \times K(G)$. Since $G$ is connected, $L_0(G)$ is connected, and consequently $K(G)$ is also connected.  $A(G)$ is also a connected Lie group with the product manifold structure. Further $A(G)$ contains $K(G)$ as a closed subgroup.
	
	\item The homogeneous space $A(G)/K(G)$ is reductive. 
\end{itemize} 

\noindent
Consider the reductive homogeneous space $A(G)/K(G)$ equipped with the canonical connection. Then, we have the following results from \cite{MK75}.
\begin{thm}
	For a connected homogeneous Lie loop $G$, the map 
	$$i:G \to A(G)/K(G), ~~~i(x) = x \times K(G)$$ is a connection preserving loop isomorphism onto $A(G)/K(G)$ with multiplication
	$$(x \times K(G)).(y \times K(G)) = (xy) \times K(G) ~~~~~ (x,y \in G)$$ 
	with respect to the canonical connections on $G$ and $A(G)/K(G)$.
\end{thm}

\noindent 
As a result, any connected homogeneous Lie loop with canonical connection can be identified with a reductive homogeneous space with canonical connection. The following result of M. Kikkawa tells us that any reductive homogeneous space with canonical connection is locally reductive. 

\begin{thm} 
	Let $S$ and $R$ denote the torsion and curvature tensors of the canonical connection $\nabla$ of a reductive homogeneous space $M=G/H$, respectively. Then $\nabla$ is locally reductive, i.e., $ \nabla S=0$ and $\nabla R = 0$.
\end{thm}

\noindent

\begin{cor}
	Any connected homogeneous Lie loop with the canonical connection is a locally reductive space.  
\end{cor}

\noindent
Below is a list of some examples of homogeneous Lie loops.

\begin{exa} 
	Any Lie group is a homogeneous Lie loop.
\end{exa}
\begin{exa}
	The set of all positive definite real symmetric matrices, denoted by $P_n$, is a homogeneous Lie loop. Loop multiplication $\mu$ being 
	$$\mu (X, Y) = X^{\frac{1}{2}}Y X^{\frac{1}{2}},~~ X, Y \in P_n.$$
\end{exa}

We are now in a position to describe a Lie-Yamaguti algebra bundle which arose from the work of M. Kikkawa.

\medskip
Since any connected homogeneous Lie loop with canonical connection is a locally reductive space, we obtain the following example (cf. \cite[Theorem $7.2$]{MK75}).

\begin{exa}\label{kikkawa}
	Let $M$ be a connected homogeneous Lie loop with the canonical connection. Let the associated torsion and curvature tensors be $S$ and $R$, respectively. Let $\xi = (TM, p, M)$ be the tangent bundle of $M.$ Define a $2$-field of brackets and a $3$-field of brackets on $M$ as follows:
	$$m\mapsto [a,b]_m = S_m(a,b); ~~ m\mapsto \{a,b,c\} = R_m(a,b) c ~~~~~ (a,b,c \in T_mG).$$
	Then $\xi$ is a Lie-Yamaguti algebra bundle.
\end{exa}

Next, we discuss a general existence theorem for locally trivial Lie-Yamaguti algebra bundle.

\begin{defn}
	Let  $(\mathfrak g, [~,~], \{~,~,~\})$ be a Lie-Yamaguti algebra and $G$ be a Lie group. We say that $G$ acts on $\mathfrak g$ if there exists a smooth homomorphism 
	$$\phi : G \to {\mbox{Aut}}_{\mbox{LYA}}(\mathfrak g), ~~g\mapsto \phi_g.$$ 
	Given such an action $\phi,$ we simply write $ga =: \phi_g(a),~g\in G,~ a \in \mathfrak g.$
\end{defn}

Note that any closed subgroup of ${\mbox{Aut}}_{\mbox{LYA}}(\mathfrak g)$  acts smoothly on $\mathfrak g$ and is a closed subgroup of the general linear group $GL_n(\mathbb R).$ 

\begin{defn}
	Let $G$ be a Lie group and $M$ a smooth manifold. A family of smooth transition maps in $M$ with values in $G$ is  an atlas 
	$\{U_i: i\in I\}$ of $M$ together with a collection of smooth maps
	$$g_{ij} : U_i \cap U_j \to G,~~ i, j \in I,$$ where $I$ is any index set which we may assume to be countable satisfying the following condition.
	\noindent
	For $i, j, k \in I,$ with $U_i\cup U_j \cup U_k \neq \emptyset,$
	$$g_{ij}(m)\cdot g_{jk}(m) = g_{ik}(m),~ m \in  U_i \cap U_j \cap U_k.$$
\end{defn}

It follows from the above condition by taking $i = j= k$ that for any $i\in I,$ $g_{ii}(m),~m \in M$ is the identity of $G.$ The above condition is known as the cocycle condition.

\medskip
We have the following existence result of locally trivial Lie-Yamguti algebra bundles whose proof is parallel to the proof of clutching construction in the theory of fibre bundles \cite{steenrod}. We outline the sketch of the proof.

\begin{thm}\label{existence-LYAB}
	Let $(\mathfrak g, [~,~], \{~, ~, ~\})$ be a Lie-Yamaguti algebra equipped with a smooth action of a Lie group $G.$ Let $M$ be a smooth manifold with a given countable atlas $\{U_i :\ i \in I\}$ together with a family of smooth transition maps
	$$g_{ij} : U_i \cap U_j \to G,~~ i, j \in I,$$
	in $M$ with values in $G.$  Then, there exists a locally trivial Lie-Yamuguti algebra bundle over $M,$ with  $\mathfrak g$ as the fibre, $G$ as the  structure group  of the bundle and with $\{g_{ij}\}$ as the associated transition maps.
\end{thm}

\begin{proof}
	Consider the following space where $I$ has the discrete topology
	$$\tilde{L} := \bigcup _{i \in I}\{(u, a,  i)| u\in U_i,~ a\in \mathfrak g, ~ i\in I\}.$$ Define an equivalence relation on $\tilde{L}$ by $(u, a, i) \sim (v, b, j)$ if and only if $u = v,~b = g_{ij}(u)a.$  Let $L = \tilde{L}/\sim.$  Let us denote the equivalence class of $(u, a, i)$ by $[u, a, i].$  Let $q : \tilde{L} \to L, (u, a,i) \mapsto [u, a, i]$ be the quotient map and $p : L \to M,~~ [u, a, i] \mapsto u$ be the natural projection map. 
	
	If $q_i = q|_{(U_i \times \mathfrak g \times \{i\})},$ then it is readily seen that $q_i$ is injective, $(q_i(U_i \times \mathfrak g \times \{i\}), q_i^{-1})$ is a smooth chart on $L$ and  $p : L \to M$ is a smooth vector bundle. 
	
	We now show that $\xi = (L, p, M)$ is a Lie-Yamaguti algebra bundle. Let $m \in M$ and $\xi_m$ be the fibre over $m.$ Define a $2$-field of brackets $m \mapsto [~, ~]_m$ and a $3$-filed of brackets $m \mapsto \{~, ~, ~\}_m$ as follows. Note that for $ i \in I,$ the map 
	$$ \{\psi_i: U_i\times \mathfrak g \rightarrow p^{(-1)}(U_i)\}$$ defined by 
	$$ \psi (u, a) = q(u, a, i),~u \in U_i,~ a \in \mathfrak g$$ gives the local trivialization of the vector bundle $\xi.$  Let $\psi_{i,m},~~m \in U_i\subset  M$ denotes the restriction of $\psi_i$ to $\{m\}\times \mathfrak g.$ 
	
	\noindent
	Let $a, b, c \in \xi_m,~m \in M.$  Choose $i \in I$ such that $m \in U_i.$ Define
	$$[a, b]_m := \psi_{i,m}([\psi_{i,m}^{-1}(a), \psi_{i, m}^{-1}(b)]),$$
	$$\{a, b, c\}_m := \psi_{i,m}(\{\psi_{i,m}^{-1}(a), \psi_{i, m}^{-1}(b), \psi_{i, m}^{-1}(c)\}).$$ Then, it is routine to verify that $\xi$ is a locally trivial Lie-Yamaguti algebra bundle with fibre $\mathfrak g.$
\end{proof}

\begin{rem}
	The above theorem provides a general method of constructing a locally trivial Lie-Yamaguti algebra bundle from any Lie group of symmetry of a given Lie-Yamaguti algebra on a manifold, equipped with a family of smooth transition maps taking values in the group of symmetry. In particular, we may apply the above method for any Lie group of symmetry of the Lie-Yamaguti algebras discussed in the previous section to construct examples of Lie-Yamaguti algebra bundles.
\end{rem}

\begin{defn}\label{morphism-LYAB}
	Let $\xi = (L, p, M)$ and $\xi^\prime = (L^\prime, p^\prime, M^\prime)$ be two Lie-Yamaguti algebra bundles.
	A homomorphism $\phi: (L, p, M) \rightarrow (L^\prime, p^\prime, M^\prime)$ from $\xi$ to $\xi^\prime$ is a vector bundle morphism $(\tilde{\phi}, \phi),$ where $\tilde{\phi}: L \rightarrow L^\prime,$ is the map between total spaces and  $\phi : M\rightarrow M^\prime$ is the map of the base spaces such that $ \tilde{\phi}|_{L_m}: L_m  \rightarrow L^\prime_{\phi (m)}$ is a Lie-Yamaguti algebra homomorphism, where $m\in M.$ 
	
	\medskip
	A homomorphism $\phi: \xi \rightarrow \xi^\prime$ of two Lie-Yamaguti algebra bundles over the same base space $M$ is a vector bundle morphism $\phi: \xi \rightarrow \xi^\prime$ such that $ \phi|_{\xi_m}: \xi_m  \rightarrow \xi^\prime_m$ is a Lie-Yamaguti algebra homomorphism for all $m \in M.$ Moreover, if $\phi|_{\xi_m}$ is a linear bijection then $\xi=(L, p, M)$ is said to be isomorphic to $\xi^\prime= (L^\prime, p^\prime, M).$
\end{defn}

\begin{defn}
	A Lie-Yamaguti algebra bundle $\xi$ is said to be trivial if it is isomorphic to a product Lie-Yamaguti algebra bundle.
\end{defn}

\section{Representation of Lie-Yamaguti Algebra Bundles}\label{$4$} 
The aim of this section is to introduce the notion of representation of Lie-Yamaguti algebra bundles.

\medskip
Our definition of representation of a Lie-Yamaguti algebra bundle is based on the definition of representation of a Lie-Yamaguti algebra \cite{KY}.

\begin{defn}
	Let $\xi = (L, p, M)$ be a Lie-Yamaguti algebra bundle and $\eta = (E, q, M)$ be a vector bundle. For any point $m\in M,$ let $\eta_m$ denote the fibre $ \eta_m = q^{-1}(m)$ of the bundle $\eta$ over $m.$
	
	\medskip
	A representation of the Lie-Yamaguti algebra bundle $\xi$ on the vector bundle $\eta$ consists of vector bundle morphisms 
	$$\rho : \xi \to \textrm{End} (\eta), ~~D,~ \theta : \xi \otimes \xi \to  \textrm{End} (\eta)$$ such that these maps restricted to each fibre satisfy the conditions (RLYB$1$) - (RLYB$6$)  as described below, where the bilinear maps 
	$$D|_{\xi_m}, ~\theta|_{\xi_m} : \xi_m\times \xi_m \to \textrm{End}(\eta_m),$$ obtained by restricting 
	$D,~\theta$  to a fibre $\xi_m$ are denoted by $D_m$ and $\theta_m,$ respectively and  similarly, 
	$\rho_m$ is the linear map $$\rho|_{\xi_m} : \xi_m \to \textrm{End}(\eta_m).$$
	For any $m \in M$ and $a, b, c, d \in \xi_m,$
	\begin{align}
		&D_m(a,b) + \theta_m(a,b) - \theta_m(b,a) = [\rho_m(a),\rho_m(b)]_m - \rho_m([a,b]); \label{RLYB1} \tag{RLYB1}
		\\
		&\theta_m(a,[b,c]_m) - \rho_m(b) \theta_m(a,c) + \rho_m(c) \theta_m(a,b) = 0; \label{RLYB2} \tag{RLYB2} \\
		&\theta_m([a,b]_m,c) - \theta_m(a,c) \rho_m(b) + \theta_m(b,c) \rho_m(a) = 0;
		\label{RLYB3} \tag{RLYB3} \\
		&\theta_m(c,d) \theta_m(a,b) - \theta_m(b,d) \theta_m(a,c)
		- \theta_m(a,\{b,c,d\}_m) + D_m(b,c) \theta_m(a,d) = 0; 
		\label{RLYB4} \tag{RLYB4} \\
		&[D_m(a,b),\rho_m(c)]_m = \rho_m(\{a,b,c\}_m);
		\label{RLYB5} \tag{RLYB5} \\
		&[D_m(a,b),\theta_m(c,d)]_m = \theta_m(\{a,b,c\}_m,d) + \theta_m(c,\{a,b,d\}_m). \label{RLYB6} \tag{RLYB6}
	\end{align}
\end{defn}

We shall denote a representation of a Lie-Yamaguti algebra bundle $\xi$ on a vector bundle $\eta$ as described above by $(\eta;~\rho,~D,~ \theta).$ A representation $(\eta;~\rho,~D,~ \theta)$ of a Lie-Yamaguti algebra bundle $\xi$ is also called a $\xi$-module.

\begin{rem}
	Like a representation of a Lie-Yamaguti algebra \cite{KY},  given a representation $(\eta;~\rho,~D,~ \theta)$ of a Lie-Yamaguti algebra bundle $\xi,$ we have for every $m\in M$
	\begin{equation}
		D_m([a, b]_m, c) + D_m([b, c]_m, a) + D_m([c, a]_m, b) = 0,\label{RLYB7} \tag{RLYB7} 
	\end{equation}
	for any $a,~ b,~ c \in \xi_m.$ 
\end{rem}

\begin{exa}\label{adjoint-representation}
	Given a  Lie-Yamaguti algebra bundle $\xi$ over $M$, we may consider $\xi$ as a $\xi$-module which gives us the adjoint representation of $\xi$ on itself. Explicitly, for each $m \in M,$ $\rho_m,~D_m,~\theta_m$ are given by
	$$ \rho_m (a): b \mapsto  [a, b]_m;~~D_m(a, b): c \mapsto \{a, b, c\}_m;~~ \theta_m (a, b) : c \mapsto \{c, a, b\}_m,$$
	for any $a,~ b,~c \in \xi_m.$ 
\end{exa}

\begin{rem}\label{restriction}
	Observe that for a $0$-dimensional manifold $M = \{pt\},$ a Lie-Yamaguti algebra bundle $\xi$ over $M$ is simply a Lie-Yamaguti algebra and a representation
	$\eta$ of $\xi$ in this case, reduces to a representation of the Lie-Yamaguti algebra $\xi.$ More generally, given any representation $(\eta; \rho, D, \theta)$ of a Lie-Yamaguti algebra bundle $\xi = (L, p, M)$ on the vector bundle $\eta = (E, q, M)$ over a smooth manifold $M,$ $(\eta_m; \rho_m, D_m, \theta_m)$ may be viewed as a representation of the Lie-Yamaguti algebra $\xi_m$ for any $m\in M.$
\end{rem}

Given a Lie-Yamaguti algebra bundle together with a representation we construct a new Lie-Yamaguti algebra bundle as follows.

\begin{exa}\label{semi-direct-bundle}
	Let  $\xi=(L, p, M)$ be a given Lie-Yamguti algebra bundle with its $2$-field of brackets and and a $3$-field of brackets denoted by
	$$m\mapsto [~, ~]_m,~~ m\in M$$ 
	$$m \mapsto \{~,~, ~\}_m,~~m \in M.$$  Let $\eta = (E,q,M)$ be a vector bundle which is a representation $(\eta;~\rho,~D,~ \theta)$ of $\xi.$ Then, $\xi \oplus \eta$ becomes a Lie-Yamaguti algebra bundle with respect to the following $2$ and $3$-fields of brackets
	\begin{equation}\label{semidirect_LYAB_rel1}
		[x+u,y+v]^{\ltimes}_m :=  [x, y]_m  + \rho_m(x) v - \rho_m(y)u
	\end{equation}
	\begin{equation} \label{semidirect_LYAB_rel2}
		\{x+u,y+v,z+w\}^{\ltimes}_m :=  \{ x,y,z \}_m + D_m(x,y) w - \theta_m(y,z) u
	\end{equation}
	for all $x, y, z \in \xi_m$ and $u,v,w \in \eta_m$. This bundle is called the semi-direct product bundle of $\xi$ and $\eta$ and is denoted by $\xi \ltimes \eta$.  
\end{exa}

Moreover, a representation of $\eta$ of a Lie-Yamaguti algebra bundle $\xi$ is characterized by the semi-direct product construction in the following sense.

\begin{prop}
	Let  $\xi=(L, p, M)$ be a given Lie-Yamaguti algebra bundle with its $2$-field of brackets and and a $3$-field of brackets denoted by
	$$m\mapsto [~, ~]_m, ~~m\in M, \quad m \mapsto \{~,~, ~\}_m,~~m \in M.$$  Let $\eta = (E,q,M)$ be a vector bundle together with vector bundle morphisms 
	$$\rho : \xi \to \textrm{End} (\eta), ~~D,~ \theta : \xi \otimes \xi \to  \textrm{End} (\eta).$$ Then,  $(\eta;~\rho,~D,~ \theta)$ is a representation of $\xi$ if and only if the Whitney sum bundle $\xi \oplus \eta$ becomes a Lie-Yamaguti algebra bundle with respect to the following $2$-fields and $3$-fields of brackets
	\begin{equation}\label{semidirect_LYAB_rel1}
		[x+u,y+v]^{\ltimes}_m :=  [x, y]_m  + \rho_m(x) v - \rho_m(y)u
	\end{equation}
	\begin{equation} \label{semidirect_LYAB_rel2}
		\{x+u,y+v,z+w\}^{\ltimes}_m :=  \{ x,y,z \}_m + D_m(x,y) w - \theta_m(y,z) u
	\end{equation}
	for all $x, y, z \in \xi_m$ and $u,v,w \in \eta_m$.
\end{prop}

\section{Cohomology of Lie-Yamaguti Algebra Bundle}\label{$5$}
In this section, we introduce cohomology of Lie-Yamaguti algebra bundle with coefficients in a representation. The definition is motivated by the definition of cohomology of a Lie-Yamaguti algebra as introduced in \cite{KY-cohomology}. We use the Remark \ref{restriction} to introduce our definition. 

\begin{defn}
	Let $\xi = (L, p, M)$ be a Lie-Yamaguti algebra bundle and  $(\eta; \rho, D, \theta)$ be a $\xi$-module. Let us denote the $2$-field and the $3$-field of brackets which make the vector bundle $\xi$ a Lie-Yamaguti algebra bundle by
	$m\mapsto [~,~]_m,~~ m \mapsto \{~, ~, ~\}_m,~~ m\in M.$
	Let $C^1(\xi; \eta) = \mbox{Hom}~(\xi; \eta)$ denote the vector space of all vector bundle maps from $\xi$ to $\eta.$ 
	Let $C^0(\xi; \eta)$ be the subspace spanned by the diagonal elements $(f,f) \in C^1(\xi;\eta) \times C^1(\xi,\eta)$. 
	For $n\geq 2,$ let $C^n( \xi; \eta)$  be the space of all vector bundle maps $f : \xi^{\otimes n} \to \eta,$ that is, $f \in\mbox{Hom}~(\xi^{\otimes n}; \eta)$ such that the resulting 
	$n$-linear maps $f_m=f|_{\xi^{\otimes n}_m} : \xi_m \times \cdots \times \xi_m \to \eta_m$ satisfy
	$f_m(x_1, \ldots, x_{2i-1}, x_{2i}, \ldots, x_n) = 0,$ if $x_{2i-1} = x_{2i},~~x_i \in \xi_m, ~i= 1, \ldots, [n/2].$ For $p \geq 1,$ set
	$$C^{(2p, 2p+1)}(\xi; \eta) :=  C^{2p}(\xi;\eta) \times C^{2p+1}(\xi; \eta).$$  Any element $(f, g) \in C^{(2p, 2p+1}(\xi; \eta)$ will be referred to as a $(2p, 2p+1)$-cochain. For $p\geq 1,$ we define a coboundary operator
	$$\delta = (\delta_I, \delta_{II}) : C^{(2p, 2p+1)}(\xi;\eta)  \to  C^{(2p+2, 2p+3)}(\xi;\eta),$$
	$$(f, g) \mapsto \delta(f, g)= (\delta_If, \delta_{II}g)$$ by defining it fibre-wise using the formula introduced by K. Yamaguti \cite{KY-cohomology}. In other words, for any $m \in M,$ 
	$$\delta (f, g)_m = ((\delta_I)_mf_m, (\delta_{II})_mg_m).$$ Explicitly, for $m \in M$ and $x_1, \ldots, x_{2p+2} \in \xi_m,$  
	\begin{align*}
		&(\delta_I)_m f_m (x_1, \ldots, x_{2p+2})\\
		&= (-1)^p [\rho_m(x_{2p+1})g_m(x_1, \ldots,x_{2p}, x_{2p+2}) - \rho_m(x_{2p+2})g_m(x_1, \ldots,x_{2p}, x_{2p+1})\\
		&\quad 
		- g_m(x_1, \ldots, x_{2p}, [x_{2p+1}, x_{2p+2}]_m)]\\
		&\quad 
		+ \sum_{k = 1}^p(-1)^{k+1} D_m(x_{2k-1}, x_{2k})f_m(x_1, \ldots, \hat{x}_{2k-1}, \hat{x}_{2k}, \ldots, x_{2p+2})\\
		&\quad  
		+ \sum_{k=1}^{p+1}\sum_{j = 2k+1}^{2p+2}(-1)^kf_m (x_1, \ldots,\hat{x}_{2k-1}, \hat{x}_{2k}, \ldots,\{x_{2k-1}, x_{2k}, x_j\}_m, \ldots, x_{2p+2}).
	\end{align*}
	Let $x_1, \ldots, x_{2p+3} \in \xi_m.$ Then,
	\begin{align*}
		& (\delta_{II})_m g_m (x_1, \ldots, x_{2p+3})\\
		&= (-1)^p[\theta_m(x_{2p+2}, x_{2p+3})g_m(x_1, \ldots, x_{2p+1})\\ 
		&\quad  
		- \theta_m(x_{2p+1}, x_{2p+3})g_m(x_1, \ldots, x_{2p}, x_{2p+2})]\\
		&\quad 
		+ \sum_{k = 1}^{p+1}(-1)^{k+1} D_m(x_{2k-1}, x_{2k})g_m(x_1, \ldots, \hat{x}_{2k-1}, \hat{x}_{2k}, \ldots, x_{2p+3})\\
		&\quad 
		+ \sum_{k=1}^{p+1}\sum_{j = 2k+1}^{2p+3}(-1)^kg_m (x_1, \ldots,\hat{x}_{2k-1}, \hat{x}_{2k}, \ldots,\{x_{2k-1}, x_{2k}, x_j\}_m, \ldots, x_{2p+3}).
	\end{align*}
\end{defn}

Now observe that for any $m\in M,$ the coboundary operator $\delta_m$ is precisely the coboundary operator for the Lie-Yamaguti algebra $\xi_m$ with coefficient in $\eta_m$ (cf. Remark \ref{restriction}) and since $\delta_m \circ \delta_m = 0$ \cite{KY-cohomology} we obtain the following result.

\begin{lem}
	For $p\geq 1,$ the coboundary operator
	$$\delta = (\delta_I, \delta_{II}) : C^{2p}(\xi;\eta) \times C^{2p+1}(\xi; \eta) \to  C^{2p+2}(\xi;\eta) \times C^{2p+3}(\xi; \eta)$$ satisfy $\delta \circ\delta = 0.$
\end{lem}

\begin{defn}
	For the case $p\geq 2,$ let $Z^{(2p, 2p+1)}(\xi; \eta)$ be the subspace of $C^{(2p, 2p+1)}(\xi;\eta)$ spanned by $(f, g)$ such that $\delta (f, g) = 0$ and $B^{(2p, 2p+1)}(\xi; \eta)$ be the subspace $\delta (C^{(2p-2, 2p-1)}(\xi;\eta)).$ Then, the $(2p, 2p+1)$-cohomology group of the Lie-Yamaguti algebra bundle $\xi$ with coefficients in $\eta$ is defined by
	$$H^{(2p, 2p+1)}(\xi; \eta) := \frac{Z^{(2p, 2p+1)}(\xi; \eta)}{B^{(2p, 2p+1)}(\xi; \eta)}.$$
\end{defn}

We next consider the case $p=1,$ and define the cohomology group $H^{(2,3)}(\xi; \eta).$ Define a coboundary operator
$$\delta = (\delta_I, \delta_{II}): C^0(\xi; \eta) \to  C^{(2, 3)} (\xi; \eta),~~ (f, f) \mapsto  (\delta_If, \delta_{II}f),$$ where for $x_1, ~x_2, ~x_3 \in \xi_m,~ m\in M,$
\begin{eqnarray*}
	(\delta_I)_m f_m (x_1,  x_2) &=& \rho_m(x_1)f_m(x_2)- \rho_m(x_2)f_m(x_1) -f_m([x_1, x_2]_m),\\ 
	(\delta_{II})_mf_m(x_1, x_2, x_3) &=& \theta_m(x_2, x_3)f_m(x_1) - \theta_m(x_1, x_3)f_m(x_2)\\
	&~&
	+ D_m(x_1, x_2)f_m(x_3)
	- f_m(\{x_1, x_2, x_3\}_m). 
\end{eqnarray*} 
Furthermore, we define another coboundary operator
$$\delta^* = (\delta^*_I, \delta^*_{II}) :  C^{(2, 3)}(\xi; \eta)  \to  C^{(3, 4)}(\xi; \eta)$$  as follows.Let $m\in M$ and $x_1, ~x_2, ~x_3~x_4 \in \xi_m.$  Then for $(f, g) \in C^{(2, 3)}(\xi; \eta),$ 
\begin{align*}
	&(\delta^*_I)_mf_m(x_1, x_2, x_3)\\
	&= \rho_m(x_1)f_m(x_2, x_3\rho_m(x_2)f_m(x_3, x_1)-\rho_m(x_3)f_m(x_1, x_2)\\
	&\quad 
	+ f_m([x_1,x_2]_m,x_3) +  f_m([x_2,x_3]_m,x_1) + f_m([x_3,x_1]_m,x_2) \\
	&\quad
	+ g_m(x_1, x_2, x_3) + g_m(x_2, x_3, x_3) + g_m(x_3, x_1, x_2),
\end{align*}
\begin{align*}
	& (\delta^*_{II})_mg_m (x_1, x_2, x_3, x_4)\\
	& = \theta_m(x_1, x_4)f_m(x_2, x_3) + \theta_m(x_2, x_4)f_m(x_3, x_1) + \theta_m(x_3, x_4)f_m(x_1, x_2)\\
	&\quad 
	+ g_m([x_1,x_2]_m,x_3, x_4) + g_m([x_2,x_3]_m,x_1, x_4) + g_m([x_3,x_1]_m, x_2, x_4).
\end{align*}
Following  \cite{KY-cohomology}, we have for each $f\in C^1(\xi; \eta)$
$$\delta_I\delta_I f = \delta^*_I\delta_If= 0~~\mbox{and}~~\delta_{II}\delta_{II}f =  \delta^*_{II}\delta_{II}f= 0.$$
In general, for $(f, g) \in C^{(2p, 2p+1)}(\xi; \eta)$ 
$$(\delta \circ \delta)(f, g) = (\delta_I\circ\delta_I(f), \delta_{II}\circ \delta_{II}(g))  =0.$$
We define
$$H^1(\xi;\eta) := \{ f \in C^1( \xi; \eta) | \delta_If = 0,~\delta_{II}f = 0\}.$$

For $p= 1$, we define the cohomology $H^{(2,3)}(\xi; \eta)$  as follows. 

\begin{defn}
	Let $Z^{(2, 3)}(\xi; \eta)$ be the subspace of $C^{(2,3)}(\xi; \eta)$  spanned by $(f, g)$ such that $\delta_If = \delta^*_If= 0,$ and  $\delta_{II}g = \delta^*_{II}g= 0.$ Let 
	$$B^{(2,3)}(\xi; \eta) = \{\delta (f, f)| f\in C^1(\xi; \eta)\}.$$
	Then, the $(2, 3)$-cohomology group of the Lie-Yamaguti algebra bundle $\xi$ with coefficients in $\eta$ is defined by
	$$H^{(2, 3)}(\xi; \eta)  = \frac{Z^{(2, 3)}(\xi; \eta)}{B^{(2,3)}(\xi; \eta)}.$$
\end{defn}

\begin{rem}
	Note that for a Lie-Yamaguti algebra bundle over a point, the above definition of cohomology groups reduces to the cohomology groups of a Lie-Yamaguti algebra as introduced by K. Yamaguti in \cite{KY-cohomology} (cf. \ref{restriction}).
\end{rem}	

\begin{rem}
	Observe that with the adjoint representation $H^1(\xi; \xi) = \mbox{Der}(\xi)$ (cf. Example \ref{derivation-cohomology}).
\end{rem}

Let $\xi = (L, p, M)$ be a Lie-Yamaguti algebra bundle with its $2$-field and the $3$-field of brackets given by
$$m\mapsto [~,~]_m,~~ m \mapsto \{~, ~, ~\}_m,~~ m\in M.$$ Let  $(\eta; \rho, D, \theta)$ be a $\xi$-module. Let $\tau = (f,g) \in Z^{(2,3)}(\xi;\eta)$ be a given cocycle. Then, we have a new Lie-Yamaguti algebra bundle as described below.

\begin{exa}\label{twisted-semi-direct-bundle}
	Consider the vector bundle $\xi \oplus \eta$ and define a $2$-field of brackets and a $3$-field of brackets as follows: For any $m \in M$
	\begin{equation}\label{Twist_LYAB_rel1}
		[x+u,y+v]^{\tau}_m := [x, y]_m + \rho_m(x) v - \rho_m(y)u + f_m(x,y)
	\end{equation}
	\begin{align} \label{Twist_LYAB_rel2}
		\{x+u,y+v,z+w\}^{\tau} &:=  \{ x,y,z\}_m + D_m(x,y) w  
		- \theta_m(y,z) u + g_m(x,y,z)
	\end{align}
	for all $x, y, z \in \xi_m$ and $u,v,w \in \eta_m$. Then, using the fact that $\tau$ is a cocycle it can be checked that equipped with these fields of brackets the  bundle $\xi \oplus \eta$ becomes a  Lie-Yamaguti algebra bundle. We call this new Lie-Yamaguti algebra bundle the twisted semi-direct product of $\xi$ and $\eta$ with respect to $\tau = (f,g)$, and is denoted by $\xi \ltimes_{\tau} \eta$. 
\end{exa}

We conclude with a remark.

\begin{rem}
	It is natural to investigate whether there is some notion of a Lie-Yamaguti algebroid for which Lie-Yamaguti algebra bundle could be viewed as a totally intransitive Lie-Yamaguti algebroid (that is, for which the anchor map is zero). We expect that a proper formulation of this notion would have a far reaching consequence in problems in geometry and physics. We would like to investigate this question in future.
\end{rem}

\bibliographystyle{amsplain}

\end{document}